\documentclass{amsart}
\usepackage{amsmath}
\usepackage{amssymb}
\usepackage{braket}
\usepackage{mathrsfs}
\usepackage{ifthen}
\usepackage{graphicx}

\usepackage{tikz}
\usetikzlibrary{patterns}
\usepackage{comment} 
\usepackage{mleftright}
\usepackage{hyperref}
\usepackage{cleveref}
\usepackage[all]{xy}
\usepackage{amscd}
\usepackage[alphabetic]{amsrefs}
\usepackage{color}
\usepackage{enumitem}
\newlist{steps}{enumerate}{1}
\setlist[steps, 1]{label = Step \arabic*:}

\usepackage[percent]{overpic}
\usepackage{graphics}
\usepackage[utf8]{inputenc} 

\makeatletter
\DeclareRobustCommand\widecheck[1]{{\mathpalette\@widecheck{#1}}}
\def\@widecheck#1#2{%
   \setbox\z@\hbox{\m@th$#1#2$}%
   \setbox\tw@\hbox{\m@th$#1%
      \widehat{%
         \vrule\@width\z@\@height\ht\z@
         \vrule\@height\z@\@width\wd\z@}$}%
   \dp\tw@-\ht\z@
   \@tempdima\ht\z@ \advance\@tempdima2\ht\tw@ \divide\@tempdima\thr@@
   \setbox\tw@\hbox{%
      \raise\@tempdima\hbox{\scalebox{1}[-1]{\lower\@tempdima\box\tw@}}}%
   {\ooalign{\box\tw@ \cr \box\z@}}}
\makeatother

\theoremstyle{plain}
\newtheorem{thm}{Theorem}[section]
\crefname{thm}{Theorem}{Theorems}
\Crefname{thm}{Theorem}{Theorems}
\newtheorem{prop}[thm]{Proposition}
\crefname{prop}{Proposition}{Propositions}
\Crefname{prop}{Proposition}{Propositions}
\newtheorem{lem}[thm]{Lemma}
\crefname{lem}{Lemma}{Lemmas}
\Crefname{lem}{Lemma}{Lemmas}
\newtheorem{cor}[thm]{Corollary}
\crefname{cor}{Corollary}{Corollaries}
\Crefname{cor}{Corollary}{Corollaries}

\crefname{claim}{Claim}{Claims}
\Crefname{claim}{Claim}{Claims}

\crefname{property}{Property}{Properties}
\Crefname{property}{Property}{Properties}

\crefname{problem}{Problem}{Problems}
\Crefname{problem}{Problem}{Problems}
\newtheorem{ques}[thm]{Question}
\crefname{ques}{Question}{Questions}
\Crefname{ques}{Question}{Questions}

\theoremstyle{definition}
\newtheorem{defn}[thm]{Definition}
\crefname{defn}{Definition}{Definitions}
\Crefname{defn}{Definition}{Definitions}
\crefname{conj}{Conjecture}{Conjectures}
\Crefname{conj}{Conjecture}{Conjectures}

\crefname{notation}{Notation}{Notations}
\Crefname{notation}{Notation}{Notations}

\crefname{convention}{Convention}{Conventions}
\Crefname{convention}{Convention}{Conventions}

\crefname{cond}{Condition}{Conditions}
\Crefname{cond}{Condition}{Conditions}

\crefname{assum}{Assumption}{Assumptions}
\Crefname{assum}{Assumption}{Assumptions}

\crefname{conj}{Conjecture}{Conjectures}
\Crefname{conj}{Conjecture}{Conjectures}

\theoremstyle{remark}
\newtheorem{rem}[thm]{Remark}
\crefname{rem}{Remark}{Remarks}
\Crefname{rem}{Remark}{Remarks}

\crefname{ex}{Example}{Examples}
\Crefname{ex}{Example}{Examples}

\crefname{section}{Section}{Sections}
\Crefname{section}{Section}{Sections}
\crefname{subsection}{Section}{Sections}
\Crefname{subsection}{Section}{Sections}
\crefname{figure}{Figure}{Figures}
\Crefname{figure}{Figure}{Figures}

\newcommand{\Z}{\mathbb{Z}}

\newcommand{\sign}{\mathrm{Sign}}

\newcommand{\s}{\mathfrak{s}}

\newcommand{\R}{\mathbb R}

\newcommand{\ctext}[1]{\raise0.2ex\hbox{\textcircled{\scriptsize{#1}}}}

\def\dim{\operatorname{dim}}

\newcommand{\mbar}[1]{{\ooalign{\hfil#1\hfil\crcr\raise.167ex\hbox{--}}}}

\title{The pretzel knot $P(4, -3, 5)$ is not squeezed}

\author{Nobuo Iida}
\address{Department of Information and Mathematical Sciences,
Division of Mathematical Sciences, Tokyo Women's Christian University,
 2-6-1 Zempukuji, Suginami-ku, Tokyo 167-8585, Japan
}
\email{n-iida8v@cis.twcu.ac.jp}

\author{Tatsumasa Suzuki}
\address{Department of Frontier Media Science, Meiji University, 4-21-1 Nakano, Nakano-ku, Tokyo 164-8525, Japan}
\email{suzukit519@meiji.ac.jp}

\date{\today}

\begin{document}

\begin{abstract}
We prove that an infinite family of three-strand pretzel knots is not squeezed.  
In particular, we show that $P(4, -3, 5)$ is not squeezed.  
This answers a question posed by Lewark (2024).  
Our proof is obtained by comparing the Rasmussen invariant with the $q_M$-invariant introduced by Taniguchi and the first author.  
\end{abstract}

\maketitle
\tableofcontents

\section{Introduction}
Knot homology has become a central topic in knot theory and provides many interesting knot invariants.  
In this paper, we study a geometric property of knots called \emph{squeezedness} for certain three-strand pretzel knots, by comparing two invariants that belong to the class of \emph{slice-torus invariants} arising from two different knot homology theories.

\subsection{Pretzel knots}
\label{subsec:pretzel}
In this section, we will review some basic materials on pretzel knots.
Let $p$, $q$, and $r$ be integers and let $P(p,q,r)$ be the three-strand pretzel link.  
For any integers $p_1, p_2, p_3$ and any permutation $\rho \in S_3$, where $S_3$ is the symmetric group on three letters, the pretzel links $P(p_{\rho(1)}, p_{\rho(2)}, p_{\rho(3)})$ and $P(p_1, p_2, p_3)$ are isotopic (see~\cite{RSuzuki2010khovanov}*{Proposition~1.2\,(ii)}). 
Moreover, for any integers $p, q$, and $r$, the reverse of the mirror $-P(p,q,r)$ are isotopic to $P(-p,-q,-r)$, where $-L$ denote the orientation reverse of the mirror of a link $L$, respectively.  

A necessary and sufficient condition for $P(p,q,r)$ to be a knot is that 
\begin{align*}
  &\text{(ODD 0)} && p, q, \text{ and } r \text{ are all odd; or}\\
  &\text{(EVEN 0)} && \text{exactly one of } p, q, \text{ or } r \text{ is even.}
\end{align*}

\subsection{Squeezedness of knots and the main theorem}
Squeezedness is a geometric property of knots in $S^3$ introduced by Feller--Lewark--Lobb \cite{Feller-Lewark-Lobb2024Squeezed}. 
Let us explain its definition.
For coprime positive integers $p$ and $q$, $T_{p,q}$ is called a \textit{positive torus knot}, and $T_{p,-q}$ is called a \textit{negative torus knot}. 
Notice that the concordance inverse of the negative torus knot $-T_{p, -q}$ is isotopic to $T_{p, q}$.
For a pair of knots $K_0, K_1\subset S^3$, a surface cobordism $\Sigma: K_0 \to K_1$ is defined to be a properly embedded oriented connected compact surface $\Sigma \subset [0, 1]\times S^3$ such that $\partial \Sigma=\{1\}\times K_1\amalg -\{0\}\times K_0$.
A knot $K$ in $S^3$ is defined to be squeezed when there is a genus minimizing surface cobordism  $\Sigma: T^- \to T^+$ \footnote{We do not assume $T^-=-T^+$. Say, $T^-=T_{p', -q'}$, $T^+=T_{p, q}$.} from a negative torus knot $T^-$ to a positive torus knot $T^+$ and an embedding $K \hookrightarrow \Sigma$ such that $\Sigma$ is the composition of two surface cobordisms $T^-\to K$ and $K \to T^+$.
It is shown in \cite{Feller-Lewark-Lobb2024Squeezed} that quasipositive knots and alternating knots are squeezed.
Lewark \cite{Lewark2024Quasipositivity} determined the squeezedness of a certain class of 3-strand pretzel knots and posed the following question: 
\begin{center}
Is $P(4, -3, 5)$ squeezed? \cite{Lewark2024Quasipositivity}
\end{center}
In this paper, we answer this question and moreover prove non-squeezedness of an infinite family of 3-strand pretzel knots including $P(4, -3, 5)$:

\begin{thm}
\label{thm:main}
For $b>0$ and $b+1 \le a \le 2b$, $P(2b+2, -(2b+1), 2a+1)$ is not squeezed. In particular, $P(4, -3, 5)$ is not squeezed.
\end{thm}
\par
This result is proved by comparing the Rasmussen invariant from Khovanov and Lee homology and the Iida--Taniguchi $q_M$-invariant \cite{Iida-Taniguchi2024Monopoles} from $\Z_2$-equivariant Seiberg--Witten monople Floer homology of double branched covers. 
\par 
We also give new computations of the $q_M$-invariant:
\begin{thm}[see Section \ref{sec:qM}]
\label{thm:qM(P(p,q,r))}
\noindent
\begin{enumerate}
\item[$(0)$] If $p=0$, then
\[
q_M(P(0, q, r)) =
\begin{cases}
\displaystyle \frac{|q+r-2|}{2}, & qr>0, \\[2mm]
\displaystyle \frac{|q-r|}{2}, & qr<0.
\end{cases}
\]
\item[$(1)$] If $p>0$, $q>0$, and $r$ are all odd, then
\[
q_M(P(p, q, r)) =
\begin{cases}
0, & \min\{p, q\} \le -r, \\[1mm]
-1, & \min\{p, q\} > -r.
\end{cases}
\]

\item[$(2)$] If $p>0$ is even and $q$ and $r$ are odd with $q \le r$, then
\[
q_M(P(p, q, r)) =
\begin{cases}
\displaystyle\frac{q+r}{2}-1, & q>0, r>0, \\[2mm]
\displaystyle\frac{q+r}{2}, & q<0, r>0, \ q+r \le 0, \\[2mm]
\displaystyle\frac{q+r}{2}, & q<0, r>0, \ q+r>0, \ p+q=1, 2q+r\le-1,\\[2mm]
\displaystyle\frac{q+r}{2}, & q<0, r>0, \ q+r>0, \ p+q<0, \\[2mm]
\displaystyle\frac{q+r}{2}, & q<0, r<0.
\end{cases}
\]
\end{enumerate}
\end{thm}


\section*{Acknowledgement}
We would like to thank Taketo Sano for contributing to his knowledge of the Rasmussen $s$-invariant. 
The first author thank Tetsuya Ito for a discussion. 
We would like to thank Hisaaki Endo for his helpful comments and his encouragement. 
We would like to thank Irving Dai and Masaki Taniguchi for a discussion on the first draft.

\section{Slice-torus invariants}
Various knot homology theories have produced a variety of knot invariants.
Some of these invariants share certain key properties and are collectively known as slice-torus invariants. 
\begin{defn}
A function
\[
f: \{\text{knots in } S^3\}/\text{concordance}\to \R
\]
is\footnote{Notice that concordance invariance implies isotopy invariance.} defined to be a \textit{slice-torus invariant} if it satisfies the following properties:
\begin{enumerate}
\item 
For any knot $K$ in $S^3$, we have $f(K) \le g_4(K)$, where $g_4(K)$ denotes the $4$-ball genus of $K$ (slice property). 

\item 
For any coprime integers $p, q > 0$, 
\[
f(T_{p,q}) = \frac{(p-1)(q-1)}{2}=g(\text{Milnor fiber})
\]
for the positive torus knot $T_{p,q}$ (torus property).

\item 
For any pair of knots $K$ and $K'$ in $S^3$, 
\[
f(K \# K') = f(K) + f(K')
\]
holds (additivity).
\end{enumerate}
\end{defn}
\begin{rem}
\begin{enumerate}
\item 
Since the connected sum $K \# -K$ is slice for any knot $K$ in $S^3$, and since a slice-torus invariant $f$ is, by definition, a concordance invariant, we have
\[
f(K) + f(-K) = f(K \# -K) = 0.
\]
Therefore,
\[
f(-K) = -f(K) \le g_4(-K) = g_4(K).
\]
Hence, for any knot $K$ in $S^3$,
\[
|f(K)| \le g_4(K).
\]
\item 
Existence of a slice-torus invariant 
gives the Milnor conjecture
\[
\frac{(p-1)(q-1)}{2}= g_4(T_{p, q}).
\]
for any coprime integers $p, q>0$.
\item 
For any surface cobordism $\Sigma: K_0 \to K_1$, 
a slice-torus invariant $f$ satisfies 
\[
|f(K_1)-f(K_0)|\leq g(\Sigma), 
\]
where $g(\Sigma)$ is the genus of the surface $\Sigma$. 
This is called the \textit{cobordism inequality} (see \cite{Lewark2014rasmussen}*{Proposition 5.1}). 
\item Neither $g_4$ nor $-\sigma/2$ is slice-torus invariant. Indeeed, $g_4$ is not additive (Consider connected sum of two copies of figure eight knots, for example). It is well-known that the knot signature is not enough to prove the Milnor conjecture.
\item 
By Section \ref{subsec:pretzel}, 
any permutation of the parameters does not change the isotopy class, that is,
\[
P(p_{\rho(1)}, p_{\rho(2)}, p_{\rho(3)}) \text{ is isotopic to } P(p_1, p_2, p_3) \quad (\rho \in S_3),
\]
and the mirror-reverse satisfies
\[
P(-p, -q, -r) \text{ is isotopic to } -P(p,q,r).
\]
It follows that
\[
f(P(p_{\rho(1)}, p_{\rho(2)}, p_{\rho(3)})) = f(P(p_1, p_2, p_3)),
\]
\[
f(P(-p, -q, -r)) = -f(P(p, q, r)).
\]
Therefore, it is sufficient to compute the slice-torus invariant $f(P(p,q,r))$ for the pretzel knot $P(p,q,r)$ with $(p,q,r)$ satisfying the following conditions:
\begin{align*}
\text{(ODD)} 
&\ \text{$p$, $q$, and $r$ are all odd with $p, q > 0$; or} \\
\text{(EVEN)} 
&\ \text{$p$ is even, and $q$ and $r$ are odd with $p \ge 0$ and $q \le r$.}
\end{align*}
\end{enumerate}
\end{rem}

Examples of slice-torus invariants include the following:
\begin{itemize}
\item 
the Ozsv\'ath--Szab\'o $\tau$-invariant from Heegaard knot Floer homology \cite{Ozsvath-Szabo2003knot};
\item
half of the Rasmussen invariant, $s/2$, from Khovanov and Lee homology \cite{Rasmussen2010khovanov};
\item
the $\mathfrak{sl}_N$ versions of the Rasmussen invariant (see \cite{Luwark-Lobb2016Quantum} and references therein);
\item
$\tau_M$, $\tau_I$, $\tau^\#_M$, and $\tau^\#_I$ from monopole and instanton Floer theory (see \cite{Ghosh-Zhenkun-Wong2024Tau} and references therein);
\item 
$\tilde{s}$ from equivariant singular instanton Floer theory \cite{Daemi-Imori-Sato-Scaduto-Taniguchi2022Instantons};
\item 
$q_M$ from $\mathbb{Z}_2$-equivariant monopole Floer theory for double branched covers \cite{Iida-Taniguchi2024Monopoles}.
\end{itemize}

Feller--Lewark--Lobb proved the following result relating squeezedness of knots and slice-torus invariants:
\begin{thm}[Feller--Lewark--Lobb \cite{Feller-Lewark-Lobb2024Squeezed}*{Lemma 3.5}]
\label{thm:Feller-Lewark-Lobb2024Squeezed}
Let $K$ in $S^3$ be a squeezed knot.
Then for any pair of slice-torus invariants $f$ and $f'$, we have
\[
f(K)=f'(K).
\]
\end{thm}
\begin{proof}
For the sake of completeness, we reproduce the proof.

Suppose $\Sigma: T^-\to T^+$ is a genus minimizing surface cobordism from a negative torus knot to a positive torus knot and it is the composition
of $\Sigma_-: T^-\to K$ and $\Sigma_+: K\to T^+$.
\footnote{We give the following remark, thought it is not necessary for the proof of the claim. Now $\Sigma_-$ and 
$\Sigma_+$ are genus minimizing as well. 
This implies
\begin{equation}
g(\Sigma_-)=g_4(K\# -T^-),\text{ and } g(\Sigma_+)=g_4(T_+\#-K).
\end{equation}
Indeed, from a connected surface in $D^4$ bounded by $K\amalg -T^-$, we can obtain a surface with the same genus in  $D^4$ bounded by $K\# -T^-$ by cutting the surface along an arc between a point in $K$ and a point in $-T^-$. On the other hand, from a connected surface in $D^4$ bounded by $K\# -T^-$, we can obtain a surface with the same genus in $D^4$ bounded by $K\amalg -T^-$ by attaching a pants cobordism. 
The argument for $\Sigma_+$ is similar.}
\par
The cobordism iniequality gives
\begin{equation}
f(K)-f(T^-)\leq g(\Sigma_-),\text{ and }f(T^+)-f(K)\leq  g(\Sigma_+)
\end{equation}
However, the equalities hold because
\[
g(\Sigma_-)+
g(\Sigma_+)=g(\Sigma)
\]
and 
\[
g(\Sigma)\geq f(T^+)-f(T^-)=g_4(T^+)+g_4(-T^-)\geq g(\Sigma)
\]
where the first inequality is the sum of the cobordism inequalities and the second inequality follows from  the assumption that $\Sigma$ is genus minimizing and thus its genus is not greater than that of the connected sum of a genus minimizing surface bounded by $T^+$ and that of $-T^-$.
Thus, we have
\[
f(K)=f(T^+)+g(\Sigma_+)=g_4(T^+)+g(\Sigma_+)
\]
which is independent of the choice of the slice-torus invariant.
\end{proof}

\section[The slice-torus invariant $q_M$]{The slice-torus invariant \texorpdfstring{$q_M$}{q\_M}}

We describe the slice-torus invariant $q_M$, introduced by Taniguchi and the first author, which arises from the $\mathbb{Z}_2$-equivariant Seiberg--Witten Floer cohomology \cite{Iida-Taniguchi2024Monopoles}.
Let $K$ be a knot in $S^3$.
The $\mathbb{Z}_2$-equivariant Seiberg--Witten Floer cohomology of the double branched cover $\Sigma_2(K)$, equipped with the unique $\mathbb{Z}_2$-invariant $\mathrm{Spin}^c$ structure $\mathfrak{s}_0$, is defined as
\[
\widetilde{H}^*_{\mathbb{Z}_2}(SWF(K))
:=
\widetilde{H}^{*+2n(\Sigma_2(K), \mathfrak{s}_0, g)}_{\mathbb{Z}_2}
\bigl(SWF(\Sigma_2(K), \mathfrak{s}_0, g); \mathbb{F}_2\bigr).
\]
This is an $\mathbb{F}_2[Q] = H^*(B\mathbb{Z}_2; \mathbb{F}_2)$-module, and it has rank one by \cite{Iida-Taniguchi2024Monopoles}*{Theorem~1.16}.

We define
\[
\mathrm{gr}_{\min,\mathrm{free}}(K)
:=
\min\{\, i \mid x \in \widetilde{H}^i_{\mathbb{Z}_2}(SWF(K)), \,
Q^n x \neq 0 \text{ for all } n \ge 0 \,\}.
\]
Then the invariant $q_M$ is given by
\[
q_M(K)
:=
-\,\mathrm{gr}_{\min,\mathrm{free}}(K)
-\frac{3}{4}\sigma(K).
\]
It is shown in \cite{Iida-Taniguchi2024Monopoles} that $q_M$ is an $\mathbb{Z}$-valued slice-torus invariant.

\par\vspace{1em}
Taniguchi and the first author proved the following:

\begin{thm}[Iida--Taniguchi {\cite{Iida-Taniguchi2024Monopoles}*{Theorem~4.5}}]
\label{thm: L with s0}
Let $K$ be a knot in $S^3$ such that
\[
\dim_{\mathbb{F}_2}\widehat{HF}(\Sigma_2(K), \mathfrak{s}_0) = 1,
\]
that is, $\Sigma_2(K)$ is an $L$-space with respect to $\mathfrak{s}_0$.
Then
\[
q_M(K) = -\frac{\sigma(K)}{2}.
\]
\qed
\end{thm}

Although this result is stated in \cite{Iida-Taniguchi2024Monopoles} under the stronger assumption that $\Sigma_2(K)$ is an $L$-space (that is, the condition holds for all $\mathrm{Spin}^c$ structures), 
the same conclusion remains valid under the weaker assumption above, without any change in the proof. 

\section[Computations of $q_M(P(p,q,r))$]{Computations of \texorpdfstring{$q_M(P(p,q,r))$}{q\_M(P(p,q,r))}}
\label{sec:qM}

\subsection[Computation of the knot signature for $P(p,q,r)$]{Computation of the knot signature for \texorpdfstring{$P(p,q,r)$}{P(p,q,r)}}

Let $p$, $q$, and $r$ be integers.  
The knot signature $\sigma(P(p,q,r))$ of the $3$-strand pretzel knot $P(p,q,r)$ was computed by Jabuka~\cite{Jabuka2010Rational}. 
Define
\[
\sign(m)
=
\begin{cases}
1, & m > 0, \\[1mm]
0, & m = 0, \\[1mm]
-1, & m < 0,
\end{cases}
\]
for any integer $m$. 

\begin{prop}[Jabuka~{\cite{Jabuka2010Rational}*{Theorem~1.18}}]
\label{prop:Jabuka-signature}
\noindent
\begin{enumerate}
\item[$(1)$] If $p$, $q$, and $r$ are all odd, then 
\[
\sigma(P(p,q,r)) = \sign(p + q) + \sign\big((p + q)(pq + qr + rp)\big). 
\]
\item[$(2)$] If $p \neq 0$ is even and $q$ and $r$ are odd, then 
\begin{align*}
\sigma(P(p,q,r))
&= -\sign(q)(|q| - 1) - \sign(r)(|r| - 1) \\
&\quad - \sign(qr(q + r)) + \sign\big((q + r)(pq + qr + rp)\big). 
\end{align*}
\end{enumerate}
\end{prop}

\begin{rem}
If $q$ and $r$ are odd, then $P(0, q, r)$ is isotopic to $T_{2,q} \# T_{2,r}$. Hence, 
\[
\sigma(P(0,q,r)) = \sigma(T_{2, q}) + \sigma(T_{2, r}) = -\sign(q)(|q| - 1) - \sign(r)(|r| - 1).
\]
\end{rem}

\begin{cor}
\label{cor: sign of P(p, q, r)}
Let $p$ be a positive even integer, $q$ a negative odd integer, and $r$ a positive odd integer satisfying $r \neq -q$.
Then
\[
\sigma(P(p, q, r)) =
\begin{cases}
-(q + r), & \displaystyle \frac{1}{p} + \frac{1}{q} + \frac{1}{r} > 0, \\[10pt]
-(q + r) + 2, & \displaystyle \frac{1}{p} + \frac{1}{q} + \frac{1}{r} < 0.
\end{cases}
\]
\end{cor}

\begin{proof}
Under the assumption
\[
\frac{1}{p} + \frac{1}{q} + \frac{1}{r} \neq 0,
\]
we have $(p + r)q \neq -pr$.

By Proposition \ref{prop:Jabuka-signature}, we obtain
\begin{align*}
\sigma(P(p, q, r))
&= (-q - 1) - (r - 1) + \sign(q + r) + \sign((q + r)(pq + qr + pr)) \\
&= -(q + r) + \sign(q + r) \left( 1 - \sign\!\left( \frac{1}{p} + \frac{1}{q} + \frac{1}{r} \right) \right),
\end{align*}
and note that
\[
\sign(pq + qr + pr)
= \sign\!\left(pqr\!\left(\frac{1}{p} + \frac{1}{q} + \frac{1}{r}\right)\!\right)
= -\sign\!\left(\frac{1}{p} + \frac{1}{q} + \frac{1}{r}\right).
\]

If $1/p + 1/q + 1/r > 0$, the final term in the last expression vanishes. 
If $1/p + 1/q + 1/r < 0$, then 
\[
\frac{q + r}{qr} = \frac{1}{q} + \frac{1}{r} < -\frac{1}{p} < 0,
\]
and hence $q + r > 0$. 
This completes the proof.
\end{proof}
\subsection[Squeezedness of $P(p,q,r)$]{Squeezedness of \texorpdfstring{$P(p,q,r)$}{P(p,q,r)}}

Let $(p, q, r)$ be integers satisfying one of the following conditions:
\begin{align*}
\text{(ODD)} 
&\ \text{$p$, $q$, and $r$ are all odd with $p, q > 0$; or} \\
\text{(EVEN)} 
&\ \text{$p$ is even, and $q$ and $r$ are odd with $p \ge 0$ and $q \le r$.}
\end{align*}

If $p$, $q$, and $r$ are all odd, then $P(p, q, r)$ is squeezed~\cite{Feller-Lewark-Lobb2024Squeezed}*{Example~2.13}. 
Therefore, if $(p, q, r)$ satisfies the condition in (ODD), then $P(p, q, r)$ is squeezed.

If $p$ is even and $q$ and $r$ are odd, then $P(p, q, r)$ is squeezed whenever $p(q + r) \le 0$ or $qr > 0$~\cite{Feller-Lewark-Lobb2024Squeezed}*{Section~4}. 
Hence, if $(p, q, r)$ satisfies the condition in (EVEN), then $P(p, q, r)$ is squeezed whenever $q + r \le 0$ or $qr > 0$.

On the other hand, if $p$ is even and $q$ and $r$ are odd, then $P(p, q, r)$ is not squeezed if $(p + q)(p + r) < 0$ \cite{Feller-Lewark-Lobb2024Squeezed}*{Section 4}. 
Therefore, if $(p, q, r)$ satisfies the condition in (EVEN), then $P(p, q, r)$ is not squeezed when $(p + q)(p + r) < 0$. 
Note that $p + q \neq 0$, $p + r \neq 0$, and $qr \neq 0$.

\begin{rem}
In general, any $2$-bridge knot is alternating (see \cite{Goodrick1972TwoBridge}*{Theorem 1}) and hence squeezed. 
If $(p, q, r)$ satisfies (EVEN) and either $|q| = 1$ or $|r| = 1$, then the pretzel knot $P(p, q, r)$ is a $2$-bridge knot, and hence it is alternating. 
Thus, in this case, $P(p, q, r)$ is squeezed.
\end{rem}

The squeezedness of $P(p, q, r)$ remains unclear only in the following case:
\begin{align*}
\text{(EVEN~X)} 
&\ \text{$p$ is even, and $q$ and $r$ are odd with } p \ge 2,\ q \le -3,\ r \ge 5, \ q + r > 0,\\
&\ \text{and } p + q > 0.
\end{align*}

\begin{lem}\label{lem:specific-classes}
If $(p, q, r)$ satisfies the condition in $\mathrm{(EVEN~X)}$, then $(p, q, r)$ can be written as
\[
(p, q, r) = (2(b + c),\ -(2b + 1),\ 2a + 1)
\]
for some positive integers $a$, $b$, and $c$ with $a > b$.
\end{lem}

\begin{proof}
If $p \ge 2$, $q \le -3$, and $r \ge 3$, then there exist positive integers $a$, $b$, and $C$ such that 
\[
p = 2C, \quad q = -(2b + 1), \quad r = 2a + 1.
\]
If $q + r > 0$, then $a > b$.  
Moreover, if $p + q > 0$, then $C \ge b + 1$.  
Hence, there exists a positive integer $c$ such that $C = b + c$.

Conversely, if $(p, q, r) = (2(b + c), -(2b + 1), 2a + 1)$ for some positive integers $a$, $b$, and $c$ with $a > b$, then the conditions in $\mathrm{(EVEN~X)}$ are satisfied.
\end{proof}

\subsection[The Rasmussen $s$-invariant of $P(p,q,r)$]{The Rasmussen \texorpdfstring{$s$}{s}-invariant of \texorpdfstring{$P(p,q,r)$}{P(p,q,r)}}

R. Suzuki~\cite{RSuzuki2010khovanov}\footnote{Note that there is a misprint in \cite{RSuzuki2010khovanov}*{Theorem 1.4}. See also \cite{KimSano2025diagrammatic}.} computed the Rasmussen $s$-invariants $s(P(p, q, r))$ for all $3$-strand pretzel knots $P(p,q,r)$ satisfying condition (ODD $0$), and for some $P(p,q,r)$ satisfying condition (EVEN $0$).  
Lewark~\cite{Lewark2014rasmussen} later calculated the $s$-invariants $s(P(p, q, r))$ for the remaining cases of $P(p,q,r)$ with condition (EVEN $0$).  
Combining these results, we obtain the following (see also \cite{KimSano2025diagrammatic}).

\begin{thm}[{\cite{RSuzuki2010khovanov}*{Theorems~1.3 and 1.4}}, {\cite{Lewark2014rasmussen}*{Theorem~4}}]
\label{thm:s(P(p,q,r))}
\noindent
\begin{enumerate}
\item[$(1)$] If $p>0$, $q>0$, and $r$ are all odd, then
\[
s(P(p, q, r)) =
\begin{cases}
0, & \min\{p, q\} \le -r, \\[1mm]
-2, & \min\{p, q\} > -r.
\end{cases}
\]

\item[$(2)$] If $p>0$ is even and $q$ and $r$ are odd with $q \le r$, then
\[
s(P(p, q, r)) =
\begin{cases}
q+r-2, & q>0, r>0, \\[1mm]
q+r, & q<0, r>0, \ q+r \le 0, \\[1mm]
q+r-2, & q<0, r>0, \ q+r>0, \ p+q>0, \\[1mm]
q+r, & q<0, r>0, \ q+r>0, \ p+q<0, \\[1mm]
q+r, & q<0, r<0.
\end{cases}
\]
\end{enumerate}
\end{thm}

For the computation of $s(P(0, q, r))$ in the case where both $q$ and $r$ are odd, see Section \ref{subsec:qM}. 

\subsection[The $q_M$-invariant of $P(p,q,r)$]{The \texorpdfstring{$q_M$}{q\_M}-invariant of \texorpdfstring{$P(p,q,r)$}{P(p,q,r)}}
\label{subsec:qM}
Taniguchi and the first author showed that
\[
\left|q_M(P(p, q, r)) + \frac{\sigma(P(p, q, r))}{2}\right| \le 1
\]
 in the proof of \cite{Iida-Taniguchi2024Monopoles}*{Theorem 1.14 (ii)}. 
In what follows, we will determine the difference
\[
q_M(P(p, q, r)) - \left(-\frac{\sigma(P(p, q, r))}{2}\right) \in \{-1, 0, 1\}
\]
for certain cases. 

\subsubsection{Some trivial cases}
Let $f$ be a slice-torus invariant.  
Let $(p,q,r)$ be integers satisfying one of the following conditions:
\begin{align*}
\text{(ODD)} 
&\ \text{$p$, $q$, and $r$ are all odd with $p, q > 0$; or}\\
\text{(EVEN)} 
&\ \text{$p$ is even, and $q$ and $r$ are odd with $p \ge 0$ and $q \le r$.}
\end{align*}

We first consider the case where the $3$-strand pretzel knot $P(p,q,r)$ is squeezed.  
In general, if a knot $K$ is squeezed, then
\[
f(K) = q_M(K) = \frac{s(K)}{2}.
\]

If $p$, $q$, and $r$ are all odd, then $P(p,q,r)$ is squeezed, and we have
\[
f(P(p,q,r)) = q_M(P(p,q,r)) = \frac{s(P(p,q,r))}{2}
=
\begin{cases}
0, & \min\{p, q\} \le -r, \\[2mm]
-1, & \min\{p, q\} > -r.
\end{cases}
\]

If $q$ and $r$ are odd, then $P(0, q, r)$ is isotopic to $T_{2,q} \# T_{2,r}$, and hence
\[
f(P(0, q, r)) = f(T_{2,q}) + f(T_{2,r}).
\]

For a negative odd integer $n$, we have
\[
f(T_{2,n}) = f(T_{2,-n}^*) = -f(T_{2,-n}),
\]
since $T_{2,n}$ is isotopic to $T_{2,-n}^*$.  
Therefore,
\[
f(T_{2,n}) =
\begin{cases}
\displaystyle \frac{n-1}{2}, & n>0, \\[2mm]
\displaystyle \frac{-n+1}{2}, & n<0,
\end{cases}
\quad\text{and hence}\quad
f(P(0, q, r)) =
\begin{cases}
\displaystyle \frac{|q+r-2|}{2}, & qr>0, \\[2mm]
\displaystyle \frac{|q-r|}{2}, & qr<0.
\end{cases}
\]

If $p>0$ is even and either $q+r \le 0$ or $qr>0$, then $P(p,q,r)$ is squeezed, and we have
\begin{align*}
f(P(p,q,r))
&= q_M(P(p,q,r)) = \frac{s(P(p,q,r))}{2}\\
&=
\begin{cases}
\displaystyle \frac{q+r}{2}-1, & q>0,\, r>0, \\[2mm]
\displaystyle \frac{q+r}{2}, & 
\text{
$
\begin{array}{l}
q<0,\, r>0,\, q+r\le0,\\
\text{or } q<0,\, r<0.
\end{array}
$
}
\end{cases}
\end{align*}

\begin{rem}
\label{rem: 2-bridge case}
If $1 \in \{|p|,|q|,|r|\}$, then $P(p,q,r)$ is a $2$-bridge knot. 
For example, $P(\pm 1, q, r)$ is isotopic to the $2$-bridge knot with Conway notation $C[q, \mp 1, r]$ (with the same choice of signs). 
Hence $P(p, q, r)$ is squeezed in this case. 
Therefore,
\[
f(P(p,q,r)) = q_M(P(p,q,r)) = \frac{s(P(p,q,r))}{2} = -\frac{\sigma(P(p,q,r))}{2}
\]
by Theorem \ref{thm:Feller-Lewark-Lobb2024Squeezed}. 
\end{rem}

\begin{rem}
\label{rem: ribbon case}
If $\{1, a, -a-4\} = \{p, q, r\}$ for some integer $a$, or $(p+q)(q+r)(r+p) = 0$, then $P(p,q,r)$ is ribbon (see \cite{Lisca2007Lens} and \cite{Lisca2007Sums} for the former cases, and \cite{GreeneJabuka2011slice}*{Theorem 1.1} and \cite{Lecuona2015slice}*{Theorem 1.1} for the latter).  
In particular, $P(p,q,r)$ is slice, and hence $g_4(P(p,q,r)) = 0$.  
Thus,
\[
|f(P(p,q,r))| \le g_4(P(p,q,r)) = 0,
\]
and consequently,
\[
f(P(p,q,r)) = q_M(P(p,q,r)) = 0.
\]
\end{rem}

Note that $q_M(P(p,q,r))$ has been completely determined when $p$, $q$, and $r$ are all odd. 

In what follows, we focus on the case of the $3$-strand pretzel knot $P(p,q,r)$ where $p$ is even and $q$ and $r$ are odd, satisfying $p \ge 2$, $q \le -3$, $r \ge 5$, and $q+r > 0$.  
Within this setting, there are two subcases depending on the sign of $p+q$:
\begin{align*}
\text{(EVEN X)} &:\ p+q > 0, \\
\text{(EVEN Y)} &:\ p+q < 0.
\end{align*}
We will consider these two subcases separately in the sequel. 

\subsection{A simple consequence of N\'emethi's graded root theory}
\label{subsec:GradedRoot}

In this section, we summarize a simple consequence of N\'emethi's graded root theory~\cite{Nemethi2005OzsvathSzabo}, which will be used in this paper.  
We only employ it to show that certain plumbed $3$-manifolds are L-spaces with respect to a given $\mathrm{Spin}^c$ structure.  
We also recall a relationship between pretzel knots and plumbing descriptions.  
For details, see \cite{Issa2018Classification} for Montesinos knots and \cite{Neumann-Raymond2006Seifert} for plumbing graphs. 

Let $M(e_0; a_1/b_1, \ldots, a_l/b_l)$ denote the Montesinos knot, where $e_0$, $a_i$, and $b_i$ are integers such that each pair $(a_i, b_i)$ is coprime.  
Note that the $3$-strand pretzel knot $P(p,q,r)$ is isotopic to $M(0; p/1, q/1, r/1)$. 

A \textit{$\tau$-sequence} $(\tau_K(n))_{n=0}^\infty$ associated with $K = M(e_0; a_1/b_1, \ldots, a_l/b_l)$, where 
$0 \le b_i < -a_i$ for all $1 \le i \le l$, and 
\[
e := e_0 - \sum_{i=1}^l\frac{b_i}{a_i} < 0,
\]
is defined by 
\[
\tau_K(0) := 0, \quad
\tau_K(n+1) := \tau_K(n) + \Delta_K(n) \quad \text{for } n \ge 0,
\]
where
\[
\Delta_K(n) := 1 - e_0 n + \sum_{i=1}^{l} \left\lfloor \frac{-b_in}{a_i} \right\rfloor
\]
for each nonnegative integer $n$.

Assume that $(p, q, r)$ satisfies $p \ge 2$, $q \le -2$, and $r \ge 2$.  
If $1/p + 1/q + 1/r > 0$, then the Montesinos knot $M(0; p/1, q/1, r/1)$ is isotopic to $M(-2; -p/(p-1), q, -r/(r-1))$.  
In this case, since $-p/(p-1) < -1$, $q < -1$, and $-r/(r-1) < -1$, a surgery diagram for the double branched cover $\Sigma_2(P(p,q,r))$ of $P(p,q,r)$ 
is represented by the weighted star-shaped graph $\Gamma$ shown in Figure \ref{fig:almost simple linear graph}.

Here we use the continued fraction expansion
\[
-\frac{s}{s-1} = [-2, \ldots, -2] := -2 - \dfrac{1}{-2 - \dfrac{1}{\ddots - \dfrac{1}{-2}}}, \quad s \ge 2,
\]
where $[-2, \ldots, -2]$ contains $s-1$ entries of $-2$. 
A graph such as $\Gamma$ is called an \textit{almost simple linear graph} (see \cite{Karakurt-Savk2019OzsvathSzabo}, \cite{Karakurt-Savk2022Almost} and \cite{TSuzuki2023OzsvathSzabo}). 

\begin{figure}[htbp]
    \centering
    \begin{overpic}[scale=0.6
    ]{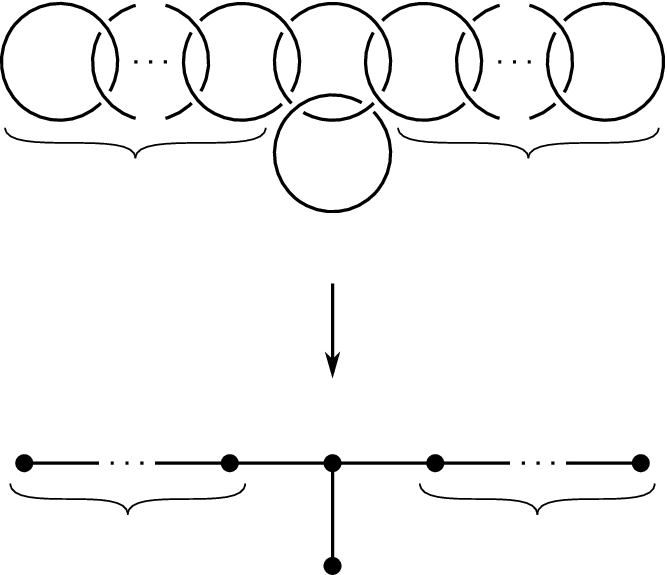}
    \put(4,87){$-2$}
    \put(31,87){$-2$}
    \put(45,87){$-2$}
    \put(60,87){$-2$}
    \put(86,87){$-2$}

    \put(14,57){$p-1$}
    \put(41,54){$q$}
    \put(73,57){$r-1$}

    \put(0,20){$-2$}

    \put(30,20){$-2$}
    \put(45,20){$-2$}
    \put(61,20){$-2$}

    \put(93,20){$-2$}
    \put(14,4){$p-1$}
    \put(44,0){$q$}
    \put(74,4){$r-1$}
\end{overpic}
\caption{A surgery diagram (top) and the corresponding plumbing graph (bottom) of $\Sigma_2(P(p,q,r))$, where $p \ge 2$, $q \le -2$, $r \ge 2$, and $1/p + 1/q + 1/r > 0$.}
\label{fig:almost simple linear graph}
\end{figure}
Since 
\[
-2+\frac{p-1}{p}+\frac{1}{-q}+\frac{r-1}{r}
= -\frac{1}{p} - \frac{1}{q} - \frac{1}{r} < 0,
\]
the plumbing graph $\Gamma$ is negative definite (see \cite{Nemethi2005OzsvathSzabo}*{Section~11.1}), 
and therefore $\Sigma_2(P(p,q,r))$
is the Seifert fibered 3–manifold with Seifert invariant
\[
(-2;\, (p, p-1),\, (-q, 1),\, (r, r-1)).
\]
Hence, we obtain 
\[
\Delta_{P(p,q,r)}(n)
= 1 + 2n 
+ \bigg\lfloor \frac{-(p-1)n}{p} \bigg\rfloor
+ \bigg\lfloor \frac{-n}{-q} \bigg\rfloor
+ \bigg\lfloor \frac{-(r-1)n}{r} \bigg\rfloor
= 1
+ \bigg\lfloor \frac{n}{p} \bigg\rfloor
+ \bigg\lfloor \frac{n}{q} \bigg\rfloor
+ \bigg\lfloor \frac{n}{r} \bigg\rfloor.
\]

If $1/p+1/q+1/r<0$, then the Montesinos knot $M(0;-p,-q,-r)$ is isotopic to $M(-1;-p,-q/(1-q),-r)$. 
In this case, since $-p<-1,-q/(1-q)<-1$, and $-r<-1$, 
a surgery diagram for the double branched cover $\Sigma_2(-P(p,q,r))$ of $-P(p,q,r)$ represented by the weighted star-shaped graph $\Gamma$ shown in Figure \ref{fig:weighted star-shaped graph}. 

\begin{figure}[htbp]
    \centering
    \begin{overpic}[scale=0.6
    ]{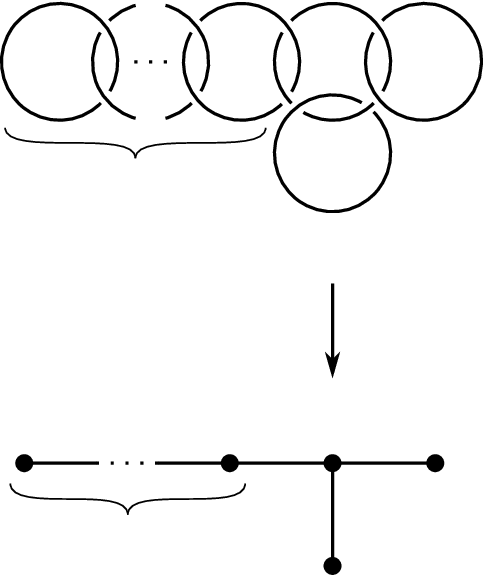}
    \put(4,102){$-2$}
    \put(38,102){$-2$}
    \put(53,102){$-1$}
    \put(70,102){$-r$}

    \put(12,66){$-q-1$}
    \put(41,64){$-p$}

    \put(0,23){$-2$}

    \put(36,23){$-2$}
    \put(53,23){$-1$}
    \put(72,23){$-r$}

    \put(12,4){$-q-1$}
    \put(47,0){$-p$}
\end{overpic}
\caption{A surgery diagram (top) and the corresponding plumbing graph (bottom) of $\Sigma_2(-P(p,q,r))$, where $p \ge 2$, $q \le -2$, $r \ge 2$, and $1/p + 1/q + 1/r < 0$. }
\label{fig:weighted star-shaped graph}
\end{figure}
Since 
\[
-1+\left(\frac{1}{p}+\frac{-q-1}{-q}+\frac{1}{r}\right)=\frac{1}{p}+\frac{1}{q}+\frac{1}{r}<0, 
\]
the plumbing graph $\Gamma$ is negative definite (see \cite{Nemethi2005OzsvathSzabo}*{Section~11.1}), 
and therefore $\Sigma_2(-P(p,q,r))$
is the Seifert fibered 3–manifold with Seifert invariant
\[
(-1;\, (p, 1), \,(-q, -q-1),\, (r, 1)).
\]
Hence, we obtain 
\[
\Delta_{P(p,q,r)}(n)
= 1 + n + \left\lfloor\frac{-n}{p}\right\rfloor 
+ \left\lfloor\frac{-(-q-1)n}{-q}\right\rfloor
+ \left\lfloor\frac{-n}{r}\right\rfloor
= 1 + \left\lfloor\frac{-n}{p}\right\rfloor
+ \left\lfloor\frac{-n}{q}\right\rfloor
+ \left\lfloor\frac{-n}{r}\right\rfloor.
\]

\begin{prop}
\label{prop:tauL}
Let $K = M(e_0; a_1/b_1, \ldots, a_l/b_l)$ be a Montesinos knot such that 
$0 \le b_i < -a_i$ for every $1 \le i \le l$, and set 
\[
e = e_0 - \sum_{i=1}^l \frac{b_i}{a_i} < 0.
\]
If the $\tau$-sequence $(\tau_K(n))_{n=0}^\infty$ of $K$ is non-decreasing, 
then the double branched cover
\[
\Sigma_2\big(M(e_0; a_1/b_1, \ldots, a_l/b_l)\big)
\]
is an L-space (with respect to all $Spin^c$ structure).
\end{prop}

\begin{proof}
According to section 11.12 of \cite{Nemethi2005OzsvathSzabo}, 
The $\tau$-sequence for general $\mathrm{Spin}^c$ structure $\s$
is also defined and given by 
\[
\tau_\s(0)=0, 
\]
\[
\tau_\s(n+1)-\tau_\s(n)=1+A_0-ne_0+\sum^l_{i=1}\bigg\lfloor \frac{-b_in+A_i}{a_i}\bigg\rfloor
\]
for some $(A_0, \dots, A_l)=(A_0(\s), \dots, A_l(\s)) \in (\Z^{\geq 0})^l$.
This implies if the $\tau$-sequence $(\tau_K(n))_{n=0}^\infty$ for $A_0=\cdots=A_l=0$ is non-decreasing, the $\tau$-sequence $(\tau_{\s}(n))_{n=0}^\infty$ for general $\s$ is non-decreasing as well.

\end{proof}

\subsection{(EVEN Y) case}
In this subsection, we focus on the case of the $3$-strand pretzel knot $P(p,q,r)$ where 
(EVEN Y): $p$ is even and $q$ and $r$ are odd, satisfying $p \ge 2$, $q \le -3$, $r \ge 5$, $q+r > 0$, and $p+q < 0$. 

\begin{thm}
\label{thm: main Y}
Suppose that $p$ is even and $q$ and $r$ are odd, satisfying 
$p \ge 2$, $q \le -3$, $r \ge 5$, $q + r > 0$, and $p + q < 0$. 
Then we have 
\[
q_M(P(p,q,r)) = \frac{q + r}{2}.
\]
\end{thm}

\begin{proof}
Under these assumptions, we have
\[
\frac{1}{p} + \frac{1}{q} + \frac{1}{r} > 0.
\]
Moreover, since $0 < p < -q$, it follows that
\[
1 + \bigg\lfloor \frac{n}{p} \bigg\rfloor + \bigg\lfloor \frac{n}{q} \bigg\rfloor 
\ge 
\bigg\lfloor \frac{n}{p} \bigg\rfloor + \bigg\lceil \frac{n}{q} \bigg\rceil
= 
\bigg\lfloor \frac{n}{p} \bigg\rfloor - \bigg\lfloor -\frac{n}{q} \bigg\rfloor 
\ge 0
\]
for all $n \ge 0$. 
Hence, by Section \ref{subsec:GradedRoot}, we obtain
\[
\Delta_{P(p,q,r)}(n)
= 
1 + \bigg\lfloor \frac{n}{p} \bigg\rfloor 
  + \bigg\lfloor \frac{n}{q} \bigg\rfloor 
  + \bigg\lfloor \frac{n}{r} \bigg\rfloor 
\ge 0
\]
for all $n \ge 0$.  
Therefore, the $\tau$-sequence $(\tau_{P(p,q,r)}(n))_{n=0}^\infty$ of the double branched cover $\Sigma_2(P(p,q,r))$ is non-decreasing.  
By Proposition~\ref{prop:tauL}, it follows that $\Sigma_2(P(p,q,r))$ is an $L$-space with respect to the unique $\mathbb{Z}_2$-invariant $\mathrm{Spin}^c$ structure $\mathfrak{s}_0$, and hence
\[
\dim_{\mathbb{F}_2} \widehat{HF}(\Sigma_2(P(p,q,r)), \mathfrak{s}_0) = 1.
\]

Applying Theorem~\ref{thm: L with s0} and Corollary~\ref{cor: sign of P(p, q, r)}, we conclude that
\[
q_M(P(p,q,r))
=
-\frac{\sigma(P(p,q,r))}{2}
=
\frac{q + r}{2}.
\]
\end{proof}

\begin{rem}
For the $3$-strand pretzel knot $P(p,q,r)$ in the case (EVEN Y), we have $p + q < 0$.  
Hence, by Theorem \ref{thm:s(P(p,q,r))}, 
\[
\frac{s(P(p,q,r))}{2} = \frac{q + r}{2}.
\]
Therefore, 
\[
q_M(P(p,q,r)) = \frac{s(P(p,q,r))}{2}
\]
holds in this case.  

However, Lewark \cite{Lewark2014rasmussen} proved that pretzel knots $P(p,q,r)$ of type (EVEN Y) are not squeezed, by comparing the Rasmussen invariant $s(P(p,q,r))$ with the Khovanov--Rozansky $\mathfrak{sl}_3$ slice-torus invariant $s_3(P(p,q,r))$.
\end{rem}

\begin{cor}
If $q$ and $r$ are odd integers satisfying $q \le -3$, $r \ge 5$, and $q + r > 0$, then 
\[
q_M(P(2,q,r)) = \frac{q + r}{2}.
\]
\end{cor}

\begin{proof}
This is a special case of Theorem \ref{thm: main Y}, since $2 + q \le -1 < 0$.
\end{proof}

\subsection{(EVEN X) case}
Finally, we will only consider the case of the $3$-strand pretzel knot with (EVEN X): $p\ge4$ is even, $q\le-3$ is odd, $r\ge5$ is odd, $q+r>0$ and $p + q > 0$. 
This case is equivalent to considering  
\[
P(2(b+c), -(2b+1), 2a+1)
\quad (a,b,c \text{ are positive integers with } a > b)
\]
by Lemma~\ref{lem:specific-classes}.  
Under this assumption, one can easily verify that
\[
\frac{1}{2(b+c)} - \frac{1}{2b+1} + \frac{1}{2a+1} \neq 0
\]
holds.  

\begin{lem}\label{lem: P(2b+2,-(2b+1),2a+1)}
For integers $b>0$ and $b+1 \le a \le 2b$, the $\tau$–sequence  
\[
(\tau_{P(2b+2,-(2b+1),2a+1)}(n))_{n=0}^\infty
\]
of the double branched cover $\Sigma_2(P(2b+2,-(2b+1),2a+1))$ is non-decreasing. 
\end{lem}

\begin{proof}
Under these assumptions, we have
\[
\frac{1}{2b+2} + \frac{1}{2b+1} + \frac{1}{2a+1} > 0.
\]
Hence, by the formula in Section~\ref{subsec:GradedRoot},  
\[
\Delta_{P(2b+2,-(2b+1),2a+1)}(n)
= 1 + \bigg\lfloor \frac{n}{2b+2} \bigg\rfloor
+ \bigg\lfloor \frac{-n}{2b+1} \bigg\rfloor
+ \bigg\lfloor \frac{n}{2a+1} \bigg\rfloor.
\]

Let $t = 2b + 1$. Then
\[
\Delta_{P(t+1,-t,2a+1)}(n)
=
\left(1 + \bigg\lfloor \frac{n}{t+1} \bigg\rfloor
+ \bigg\lfloor \frac{n}{2a+1} \bigg\rfloor \right)
- \bigg\lceil \frac{n}{t} \bigg\rceil.
\]
Denote by $Q_{n,t}$ and $R_{n,t}$ the quotient and remainder when $n$ is divided by $t$, respectively.  
Set
\[
A_{n,t} = \bigg\lfloor \frac{n}{t+1} \bigg\rfloor, \quad
B_{n,a} = \bigg\lfloor \frac{n}{2a+1} \bigg\rfloor, \quad
C_{n,t} = \bigg\lceil \frac{n}{t} \bigg\rceil.
\]
Note that $A_{n,t}, B_{n,a}, C_{n,t} \ge 0$ for $n \ge 0$. Then
\[
\Delta_{P(t+1,-t,2a+1)}(n)
= A_{n,t} + B_{n,a} + 1 - C_{n,t}.
\]

We now consider several cases.

\medskip
\noindent
\textbf{Case 1.} $Q_{n,t} = 0$.  
Then $0 \le n \le t-1$, and hence $C_{n,t} \le 1$.  
Thus $\Delta_{P(t+1,-t,2a+1)}(n) \ge 1 - C_{n,t} \ge 0$.

\medskip
\noindent
\textbf{Case 2.} $Q_{n,t} = 1$, $R_{n,t} = 0$.  
Then $n = t$, so $C_{n,t} = 1$, and again $\Delta_{P(t+1,-t,2a+1)}(n) \ge 0$.

\medskip
\noindent
\textbf{Case 3.} $Q_{n,t} = 1$, $R_{n,t} > 0$.  
Then $t+1 \le n \le 2t-1$, and $(A_{n,t}, C_{n,t}) = (1, 2)$.  
Hence $\Delta_{P(t+1,-t,2a+1)}(n) \ge A_{n,t} + 1 - C_{n,t} \ge 0$.

\medskip
\noindent
\textbf{Case 4.} $2 \le Q_{n,t} \le t$.  
Since $Q_{n,t}t \le n \le (Q_{n,t}+1)t - 1$, we have
$A_{n,t} \ge Q_{n,t}-1$ and $C_{n,t} \le Q_{n,t}+1$.  
If $b>0$ and $b+1 \le a \le 2b$, then $Q_{n,t}t \ge 2t = 4b+2 > 2a+1$,  
so $B_{n,a} \ge 1$, and therefore
\[
\Delta_{P(t+1,-t,2a+1)}(n) \ge (Q_{n,t}-1) + 1 + 1 - (Q_{n,t}+1) = 0.
\]

\medskip
\noindent
\textbf{Case 5.} $Q_{n,t} \ge t+1$ and $a \ge 3$.  
From $Q_{n,t} \ge t+1$ and $2b \ge a$, we have
\begin{align*}
\Delta_{P(t+1,-t,2a+1)}(n)
&\ge \frac{n-t}{t+1} + \frac{n-2a}{2a+1} + 1 - \frac{n+t-1}{t}\\
&= \left( \frac{1}{t+1} + \frac{1}{2a+1} - \frac{1}{t} \right)n
- \frac{t}{t+1} - \frac{2a}{2a+1} + \frac{1}{t}\\
&\ge t + \frac{t(t+1)}{2a+1} - (t+1)
- \frac{t}{t+1} - \frac{2a}{2a+1} + \frac{1}{t}\\
&= \frac{(2b+1)(2b+2)}{2a+1} - 1 - \frac{2b+1}{2b+2}
- \frac{2a}{2a+1} + \frac{1}{2b+1}\\
&\ge \frac{(a+1)(a+2)}{2a+1} - 1 - \frac{2b+1}{2b+2}
- \frac{2a}{2a+1} + \frac{1}{2b+1}\\
&> \frac{a^2 + a + 2}{2a+1} - 2
= \frac{a(a-3)}{2a+1} > 0.
\end{align*}

\medskip
\noindent
\textbf{Case 6.} $Q_{n,t} \ge t+1$ and $a = 2$.  
Then $b = 1$, $t = 3$, and $Q_{n,3} \ge 4$. We obtain
\begin{align*}
\Delta_{P(t+1,-t,2a+1)}(n)
&\ge \frac{n-3}{4} + \frac{n-4}{5} + 1 - \frac{n-2}{3}\\
&\ge \frac{12}{5} - 1 - \frac{3}{4} - \frac{4}{5} + \frac{1}{3} > 0.
\end{align*}

Therefore, the $\tau$-sequence
\[
(\tau_{P(2b+2,-(2b+1),2a+1)}(n))_{n=0}^\infty
\]
of $\Sigma_2(P(2b+2,-(2b+1),2a+1))$ is non-decreasing.
\end{proof}

We now prove Theorem \ref{thm:main}. 
\begin{proof}[Proof of Theorem \ref{thm:main}]
If $b>0$ and $b+1 \le a \le 2b$, then by Lemma~\ref{lem: P(2b+2,-(2b+1),2a+1)}, the $\tau$-sequence $(\tau(n))_{n=0}^\infty$ of $\Sigma_2(P(2b+2,-(2b+1),2a+1))$ is non-decreasing. 
Therefore, the double branched cover $\Sigma_2(P(2b+2,-(2b+1),2a+1))$ is an $L$-space with respect to the unique $\mathbb{Z}_2$-invariant $\mathrm{Spin}^c$ structure $\s_0$. 
Hence,
\[
\dim_{\mathbb{F}_2}\,\widehat{HF}(\Sigma_2(P(2b+2,-(2b+1),2a+1)),\s_0)=1.
\]
Since $1/(2b+2)-1/(2b+1)+1/(2a+1)>0$ in this case, we have
\begin{align*}
q_M(P(2b+2,-(2b+1),2a+1))
&= -\frac{\sigma(P(2b+2,-(2b+1),2a+1))}{2} \\
&= \frac{-(2b+1)+(2a+1)}{2} \\
&= a-b
\end{align*}
by Theorem~\ref{thm: L with s0} and Corollary~\ref{cor: sign of P(p, q, r)}. 

Moreover, since $-(2b+1)<0$, $2a+1>0$, $-(2b+1)+2a+1=2(a-b)>0$, and $2b+2-(2b+1)>0$, Theorem~\ref{thm:s(P(p,q,r))} yields
\[
\frac{s(P(2b+2,-(2b+1),2a+1))}{2}=a-b-1.
\]
Therefore, by the contrapositive of Theorem~\ref{thm:Feller-Lewark-Lobb2024Squeezed}, the pretzel knot 
\[
P(2b+2,-(2b+1),2a+1)
\]
is not squeezed. 
\end{proof}

\begin{rem}
By \cite{Feller-Lewark-Lobb2024Squeezed}*{Proposition 1.2} and Theorem \ref{thm:main}, 
the knot $P(2b+2,-(2b+1),2a+1)$ is not quasi-alternating for any integers $a$ and $b$ satisfying $b>0$ and $b+1 \le a \le 2b$. Notice that quasi-alternatingness of Mentesions knots are completely determined in \cite{Issa2018Classification}*{Theorem 1}, so this result is not new.

Moreover, by \cite{Waite2020Three}*{Corollary 6.9}, we have
\[
\tau(P(2a,-(2b+1),2c+1)) = c - b - 1
\]
for any integers $a$, $b$, and $c$ satisfying $\min\{a, c\} > b > 0$. 

Hence, we have 
\begin{align*}
\tau(P(2b+2,-(2b+1),2a+1))
&= \frac{s(P(2b+2,-(2b+1),2a+1))}{2} \\[1mm]
&= -\frac{\sigma(P(2b+2,-(2b+1),2a+1))}{2} \\[1mm]
&= a - b - 1
\end{align*}
for any integers $a$ and $b$ satisfying $b>0$ and $b+1 \le a \le 2b$. 
Thus our determination of non-squeezedness cannot be recovered by comparering $s/2$ and $\tau$. 

From \cite{Manolescu-Ozsvath2007Khovanov}*{Theorem 2}, if $K$ is a quasi-alternating knot, then
\[
\tau(K) = \frac{s(K)}{2} = -\frac{\sigma(K)}{2}.
\]
It follows that the slice-torus invariant $q_M$ can detect infinitely many knots 
for which the converse of the above statement does not hold.
\end{rem}

Finally, we introduce a conjecture. 

\begin{ques}\label{ques:main}
If $p\ge4$ is even, $q\le-3$ is odd, $r\ge5$ is odd, and $p+q>0$, and $q+r>0$, 
then does
\[
q_M(P(p, q, r))=-\frac{\sigma(P(p, q, r))}{2}. 
\]
hold?
\end{ques}

\begin{rem}\label{rem: q_M and squeezedness}
If the answer is yes, then we have
\[
q_M(P(p, q, r))=\displaystyle\frac{q+r}{2}\neq\displaystyle\frac{q+r}{2}-1=\frac{s(P(p, q, r))}{2} 
\]
if $1/p + 1/q + 1/r > 0$ by Corollary \ref{cor: sign of P(p, q, r)} and Theorem \ref{thm:s(P(p,q,r))}, and thus $P(p, q, r)$ is not squeezed in this case. 
If the answer is yes, we also have
\[
q_M(P(p, q, r))=\displaystyle\frac{q+r}{2}-1=\frac{s(P(p, q, r))}{2} 
\]
if $1/p + 1/q + 1/r < 0$ by Corollary \ref{cor: sign of P(p, q, r)} and Theorem \ref{thm:s(P(p,q,r))}. 
\end{rem}

\bibliographystyle{amsalpha}
\bibliography{paper}

@article{Daemi-Imori-Sato-Scaduto-Taniguchi2022Instantons,
  title={Instantons, special cycles, and knot concordance},
  author={Daemi, Aliakbar and Imori, Hayato and Sato, Kouki and Scaduto, Christopher and Taniguchi, Masaki},
  journal={arXiv preprint arXiv:2209.05400},
  year={2022}
}

@article{Luwark-Lobb2016Quantum,
  title={New quantum obstructions to sliceness},
  author={Lewark, Lukas and Lobb, Andrew},
  journal={Proceedings of the London Mathematical Society},
  volume={112},
  number={1},
  pages={81--114},
  year={2016},
  publisher={Oxford University Press}
}

@article{Ghosh-Zhenkun-Wong2024Tau,
author = {Ghosh, Sudipta and Li, Zhenkun and Wong, C.‐M},
year = {2024},
pages = {},
title = {On the tau invariants in instanton and monopole Floer theories},
volume = {17},
journal = {Journal of Topology},
doi = {10.1112/topo.12346}
}

@article{Iida-Taniguchi2024Monopoles,
  title={Monopoles and transverse knots},
  author={Iida, Nobuo and Taniguchi, Masaki},
  journal={arXiv preprint arXiv:2403.15763},
  year={2024}
}

@article{Nemethi2005OzsvathSzabo,
  title={On the Ozsv{\'a}th-Szab{\'o} invariant of negative definite plumbed 3-manifolds},
  author={N{\'e}methi, Andr{\'a}s},
  journal={Geometry \& Topology},
  volume={9},
  number={2},
  pages={991--1042},
  year={2005},
  publisher={Mathematical Sciences Publishers},
  shorthand = {Nem05}
}

@article{Feller-Lewark-Lobb2024Squeezed,
  title={Squeezed knots},
  author={Feller, Peter and Lewark, Lukas and Lobb, Andrew},
  journal={Quantum Topology},
  year={2024},
  shorthand = {FLL24}
}

@article{Jabuka2010Rational,
  title={Rational Witt classes of pretzel knots},
  author={Jabuka, Stanislav},
  year={2010}
}

@article{RSuzuki2010khovanov,
  title={Khovanov homology and Rasmussen's s-invariants for pretzel knots},
  author={Suzuki, Ryohei},
  journal={Journal of Knot Theory and Its Ramifications},
  volume={19},
  number={09},
  pages={1183--1204},
  year={2010},
  publisher={World Scientific}
}

@article{Lewark2014rasmussen,
  title={Rasmussen's spectral sequences and the {$\mathfrak{sl}_n$}-concordance invariants},
  author={Lewark, Lukas},
  journal={Advances in mathematics},
  volume={260},
  pages={59--83},
  year={2014},
  publisher={Elsevier}
}

@misc{KimSano2025diagrammatic,
      title={A diagrammatic approach to the Rasmussen invariant via tangles and cobordisms}, 
      author={KeeTaek Kim and Taketo Sano},
      year={2025},
      eprint={2503.05414},
      archivePrefix={arXiv},
      primaryClass={math.GT},
      url={https://arxiv.org/abs/2503.05414}, 
}

@article{Lewark2024Quasipositivity,
   title={Quasipositivity and braid index of pretzel knots},
   volume={32},
   ISSN={1944-9992},
   url={http://dx.doi.org/10.4310/CAG.241121001754},
   DOI={10.4310/cag.241121001754},
   number={5},
   journal={Communications in Analysis and Geometry},
   publisher={International Press of Boston},
   author={Lewark, Lukas},
   year={2024},
   pages={1435–1444} }

@article{GreeneJabuka2011slice,
  title={The slice-ribbon conjecture for 3-stranded pretzel knots},
  author={Greene, Joshua and Jabuka, Stanislav},
  journal={American journal of mathematics},
  volume={133},
  number={3},
  pages={555--580},
  year={2011},
  publisher={Johns Hopkins University Press}
}

@article{Lecuona2015slice,
  title={On the slice-ribbon conjecture for pretzel knots},
  author={Lecuona, Ana G},
  journal={Algebraic \& Geometric Topology},
  volume={15},
  number={4},
  pages={2133--2173},
  year={2015},
  publisher={Mathematical Sciences Publishers}
}

@article{Goodrick1972TwoBridge,
  title={Two bridge knots are alternating knots},
  author={Goodrick, Richard},
  journal={Pacific Journal of Mathematics},
  volume={40},
  number={3},
  pages={561--564},
  year={1972},
  publisher={Mathematical Sciences Publishers}
}

@article{Ozsvath-Szabo2003knot,
  title={Knot Floer homology and the four-ball genus},
  author={Ozsv{\'a}th, Peter and Szab{\'o}, Zolt{\'a}n},
  journal={Geometry \& Topology},
  volume={7},
  number={2},
  pages={615--639},
  year={2003},
  publisher={Mathematical Sciences Publishers}
}

@article{Rasmussen2010khovanov,
  title={Khovanov homology and the slice genus},
  author={Rasmussen, Jacob},
  journal={Inventiones mathematicae},
  volume={182},
  number={2},
  pages={419--447},
  year={2010},
  publisher={Springer}
}

@article{Lisca2007Lens,
  title={Lens spaces, rational balls and the ribbon conjecture},
  author={Lisca, Paolo},
  journal={Geometry \& Topology},
  volume={11},
  number={1},
  pages={429--472},
  year={2007},
  publisher={Mathematical Sciences Publishers}
}

@article{Lisca2007Sums,
  title={Sums of lens spaces bounding rational balls},
  author={Lisca, Paolo},
  journal={Algebraic \& Geometric Topology},
  volume={7},
  number={4},
  pages={2141--2164},
  year={2007},
  publisher={Mathematical Sciences Publishers}
}

@article{Manolescu-Ozsvath2007Khovanov,
  title={On the Khovanov and knot Floer homologies of quasi-alternating links},
  author={Manolescu, Ciprian and Ozsv{\'a}th, Peter},
  journal={arXiv preprint arXiv:0708.3249},
  year={2007}
}

@phdthesis{Waite2020Three,
  title={Three-strand pretzel knots, knot Floer homology and concordance invariants},
  author={Waite, Daniel},
  year={2020},
  school={University of Glasgow}
}

@article{Karakurt-Savk2019OzsvathSzabo,
  title={Ozsvath-Szabo {$d$}-invariants of almost simple linear graphs},
  author={Karakurt, Cagri and Savk, Oguz}, 
  journal={arXiv preprint arXiv:1911.01688},
  year={2019}
}

@article{Karakurt-Savk2022Almost,
  title={Almost simple linear graphs, homology cobordism and connected Heegaard Floer homology},
  author={Karakurt, Cagri and Savk, Oguz}, 
  journal={Acta Mathematica Hungarica},
  volume={168},
  number={2},
  pages={454--489},
  year={2022},
  publisher={Springer}
}

@article{Issa2018Classification,
  title={The classification of quasi-alternating Montesinos links},
  author={Issa, Ahmad},
  journal={Proceedings of the American Mathematical Society},
  volume={146},
  number={9},
  pages={4047--4057},
  year={2018}
}

@inproceedings{Neumann-Raymond2006Seifert,
  title={Seifert manifolds, plumbing, $\mu$-invariant and orientation reversing maps},
  author={Neumann, Walter D and Raymond, Frank},
  booktitle={Algebraic and Geometric Topology: Proceedings of a Symposium held at Santa Barbara in honor of Raymond L. Wilder, July 25--29, 1977},
  pages={163--196},
  year={2006},
  organization={Springer}
}

@article{TSuzuki2023OzsvathSzabo,
  title={On the Ozsv\'ath-Szab\'o {$d$}-invariants for almost simple linear graphs},
  author={Suzuki, Tatsumasa},
  journal={arXiv preprint arXiv:2310.14279},
  year={2023}
}
\end{document}